\documentclass[12pt]{amsart}
\usepackage{amscd,amssymb,times,fullpage, }
\usepackage{graphics, color}
\newtheorem{theorem}{Theorem}%[section]
\newtheorem{lemma}[theorem]{Lemma}
\newtheorem{proposition}[theorem]{Proposition}
\newtheorem{corollary}[theorem]{Corollary}
\theoremstyle{definition}
\newtheorem{definition}[theorem]{Definition}

\newtheorem{example}[theorem]{Example}

\theoremstyle{remark}
\newtheorem{remark}[theorem]{Remark}
\newcommand{\HHH}{\mathcal{ H}}

\newcommand{\FF}{\mathcal{ F}}

\def\bZ{{\mathbb Z}}

\def\bR{{\mathbb R}}
\def\bH{{\mathbb H}}
\def\bS{{\mathbb S}}

\newcommand\Cont{\operatorname{Cont}}

\newcommand\Ker{\operatorname{Ker}}

\newcommand\Diff{\operatorname{Diff}}

\catcode`\@=\active % make @ other character ( is default)

\begin{document}

\title{On normal contact pairs}
\author{Gianluca~Bande}
\address{Dipartimento di Matematica e Informatica, Universit{\`a} degli studi di Cagliari, Via Ospedale 72, 09124 Cagliari, Italy}
\email{gbande{\char'100}unica.it}
\author{Amine~Hadjar}
\address{Laboratoire de Math{\'e}matiques, Informatique et
Applications, Universit{\'e} de Haute Alsace - 4 Rue de Fr{\`e}res
Lumi{\`e}re, 68093 Mulhouse Cedex, France}
\email{mohamed.hadjar{\char'100}uha.fr}

\thanks{The first author was supported by the MIUR Project: \textit{ Riemannian metrics and differentiable manifolds}--P.R.I.N. $2005$ and by a Visiting Professor fellowship at Universit\'e de Haute Alsace - Mulhouse in august 2007.}
\subjclass[2000]{Primary 53C15; Secondary 53D15, 53C12, 53D35.}
\keywords{Complex structure; almost contact structure; contact
pair; foliation}
\begin{abstract}
We consider manifolds endowed with a contact pair structure. To
such a structure are naturally associated two almost complex
structures. If they are both integrable, we call the structure a
normal contact pair. We generalize the Morimoto's Theorem on
product of almost contact manifolds to flat bundles. We construct
some examples on Boothby--Wang fibrations over contact-symplectic
manifolds. In particular, these results give new methods to
construct complex manifolds.
\end{abstract}

\maketitle

\section{Introduction}

A contact pair on a manifold is a pair of one-forms $\alpha_1$ and
$\alpha_2$ of constant and complementary classes, for which
$\alpha_1$ induces a contact form on the leaves of the
characteristic foliation of $\alpha_2$, and vice versa. This
notion, considered in \cite{Bande1,BH}, was firstly introduced in \cite{Blair2} by the name \textit{bicontact} and further studied in \cite{Abe}.

In \cite{BH2} we considered the notion of contact pair structure on a manifold $M$, that is a
contact pair $(\alpha_1, \alpha_2)$ together with a tensor field $\phi$ on $M$, of type $(1,1)$, such
 that $\phi^2=-Id + \alpha_1 \otimes Z_1 + \alpha_2 \otimes Z_2$ and $\phi(Z_1)=\phi(Z_2)=0$, where $Z_1$ and $Z_2$
are the Reeb vector fields of the pair. This is a special type of \textit{$f$-structure with complemented frame} (see \cite{blair3, nakagawa, yano}). 

In this paper, we associate to a contact pair structure the almost
complex structures  $J=\phi  - \alpha_2 \otimes Z_1 + \alpha_1
\otimes Z_2$ and $T=\phi  + \alpha_2 \otimes Z_1 - \alpha_1
\otimes Z_2$.
 This can be seen as a generalization of the almost
complex structure used in almost contact geometry to define
normality (see \cite{Blairbook} and the references therein).
Nevertheless our structure is more intrinsic in that, for its
definition, we do not need to consider the manifold $M \times \bR$
as in the case of the almost contact structures. A natural problem
is the study of the integrability condition for these almost
complex structures and we call a contact pair structure
\textit{normal}, if the associated almost complex structures are
both integrable. An interesting feature of this structure is that,
under the assumption that $\phi$ is decomposable, there are almost
contact structures induced on the leaves (which are contact
manifolds) of the characteristic foliations, and then a natural
problem is to relate the normality of the whole structure to that
of the induced structures (in the sense of almost contact
manifolds).

One could expect a general result similar to that of Morimoto
\cite{Mori}, which says that on a product of manifolds, each of
them endowed with an almost contact structures, there is a natural
almost complex structure which is integrable if and only if the
almost contact structures are normal.

In our case this is not true in full generality, since there are interesting
counterexamples showing that the contact pair structure $(\alpha_1, \alpha_2 , \phi)$ on $M$ can be more
complicated: even if $M$ is locally the product of two contact manifolds, the tensor field $\phi$ is not the sum of two tensors on the factors.

Anyway, we can generalize Morimoto's result in the
context of flat bundles, already used in \cite{KM, BK} to construct new examples of symplectic pairs.

By performing the Boothby--Wang fibration over a manifold endowed
with a contact-symplectic pair \cite{Bande2} (which can be thought
as a special almost contact structure), we are able to construct
$\bS^1$-invariant contact pair structures on the total space and
we show that under some hypothesis, the contact pair structure is
normal if and only if the contact-symplectic pair on the base is
normal as almost contact structure. This is an even counterpart of
the constructions given by Morimoto (resp. Hatakeyama) of normal
contact structures on Boothby--Wang fibration over a complex (resp. almost K\"ahler) manifold.

Furthermore, the flat bundles and the Boothby--Wang fibrations
yielding normal contact pairs give new constructions of complex
manifolds.

In the sequel we denote by $\Gamma(B)$ the space of sections of a
vector bundle $B$. For a given foliation $\mathcal{F}$ on a
manifold $M$, we denote by $T\mathcal{F}$ the subbundle of $TM$
whose fibers are given by the distribution tangent to the leaves.
All the differential objects considered are supposed to be smooth.

\section{Preliminaries on contact pairs and contact pair structures}\label{s:prelim}

In this section we firstly give the notions concerning contact pairs which are useful for our purpose, next we recall the definition and the properties of contact pair structures. A manifold endowed with a contact pair was called \textit{bicontact} in \cite{Blair2}. Here we maintain the notations of \cite{Bande1, BH} and we refer to \cite{Bande1, Bande2, BH, BH2, BGK, BK} for further informations and several examples of such structures.
\begin{definition}[\cite{Bande1, BH, Blair2}]\label{d:cpair}
A pair $(\alpha_1, \alpha_2)$ of $1$-forms on a
$(2h+2k+2)$-dimensional manifold $M$ is said to be a contact pair
of type $(h,k)$ if:
\begin{enumerate}
\item[i)] $\alpha_1\wedge (d\alpha_1)^{h}\wedge\alpha_2\wedge
(d\alpha_2)^{k}$ is a volume form, \item[ii)]
$(d\alpha_1)^{h+1}=0$ and $(d\alpha_2)^{k+1}=0$.
\end{enumerate}
\end{definition}
Since the form $\alpha_1$ (resp. $\alpha_2$) has constant class $2h+1$
(resp. $2k+1$), the distribution $\Ker \alpha_1 \cap \Ker d\alpha_1$ (resp.
 $\Ker \alpha_2 \cap \Ker d\alpha_2$) is completely integrable and then it determines the so-called characteristic
  foliation $\mathcal{F}_1$ (resp. $\mathcal{F}_2$) whose leaves are endowed with a contact form induced by $\alpha_2$ (resp. $\alpha_1$).\\

To a contact pair $(\alpha_1, \alpha_2)$ of type $(h,k)$ are
associated two commuting vector fields $Z_1$ and $Z_2$, called \textit{Reeb
vector fields} of the pair, which are uniquely determined by the
following equations:
\begin{eqnarray*}
&\alpha_1 (Z_1)=\alpha_2 (Z_2)=1  , \; \; \alpha_1
(Z_2)=\alpha_2
(Z_1)=0 \, , \\
&i_{Z_1} d\alpha_1 =i_{Z_1} d\alpha_2 =i_{Z_2}d\alpha_1=i_{Z_2}
d\alpha_2=0 \, ,
\end{eqnarray*}
where $i_X$ is the contraction with the vector field $X$.
In particular, since the Reeb vector fields commute, they determine a locally free
$\mathbb{R}^2$-action, called  the \textit{Reeb action}.

The kernel distribution of $d\alpha_1$ (resp. $d\alpha_2$) is
also integrable and then it defines a foliation whose leaves
inherit a contact pair of type $(0,k)$ (resp. $(h,0)$).

%A contact pair of type $(h,k)$ has a local model (see \cite{BH, Bande2}), which means that for every point of the manifold there exists a coordinate
%chart on which the pair can be written as follows:
%$$
%\alpha_1=dx_{2h+1}+\sum_{i=1}^{h}x_{2i-1}dx_{2i}  ,\; \;
%\alpha_2=dy_{2k+1}+\sum_{i=1}^{k}y_{2i-1}dy_{2i} \, ,
%$$
%where $(x_1, \cdots ,x_{2h+1}, y_1 , \cdots , y_{2k+1})$ are the standard coordinates on $\mathbb{R}^{2h+2k+2}$.\\
The tangent bundle of a manifold $M$ endowed with a contact pair
can be split in different ways. For $i=1,2$, let $T\mathcal F _i$
be the subbundle determined by the characteristic foliation of
$\alpha_i$, $T\mathcal G_i$ the subbundle of $TM$ whose fibers are
given by $\ker d\alpha_i \cap \ker \alpha_1 \cap \ker \alpha_2$
and $\mathbb{R} Z_1, \mathbb{R} Z_2$ the line bundles determined
by the Reeb vector fields. Then:
\begin{align}
\label{split-CP1} TM &=T\mathcal F _1 \oplus T\mathcal F _2  \\
\label{split-CP2} TM &=T\mathcal G_1 \oplus T\mathcal G_2 \oplus \mathbb{R} Z_1 \oplus \mathbb{R} Z_2 .
\end{align}
Moreover we have $T\mathcal F _1=T\mathcal G_1 \oplus \mathbb{R} Z_2 $ and $T\mathcal F _2=T\mathcal G_2 \oplus \mathbb{R} Z_1 $.

In a similar way, we define symplectic pairs and contact-symplectic pairs:
\begin{definition}[\cite{BK}]
A symplectic pair of type $(h, k)$, for $h, k \neq 0$, on a
$2h+2k$-dimensional manifold $M$ is a pair of closed two-forms
$\omega_{1}$, $\omega_{2}$ such that:
\begin{itemize}
\item[i)] $\omega_{1}^h \wedge \omega_{2}^{k}$ is a volume form;
\item[ii)] $\omega_{1}^{h+1}=0$ and $\omega_{2}^{k+1}=0$.
\end{itemize}
\end{definition}
\begin{definition}[\cite{Bande1,Bande2}]
A contact-symplectic pair of type $(h,k)$ on a $(2h+2k+1)$-dimensional manifold $N$ consists of a $1$-form $\beta$ and a closed $2$-form $\eta$ such that:
\begin{enumerate}
\item[i)] $\beta\wedge (d \beta)^{h} \wedge \eta^{k}$ is a volume
form,
\item[ii)] $(d\beta)^{h+1}=0$ and $\eta^{k+1}=0$.
\end{enumerate}
\end{definition}
To a contact-symplectic pair is associated a Reeb vector field
$W$, uniquely defined by the following equations:
\begin{equation}
\beta (W)=1 \, \, , \, \, i_W d\beta= i_W \eta=0 \, .
\end{equation}
Furthermore, let $\FF_1$ and $\FF_2$ be the characteristic foliations of $\eta$ and $\beta$ respectively, and $T\FF_1$, $T\FF_2$
 the corresponding subbundles of $TN$. Let $\bR W$ the line bundle determined by the Reeb vector field and $T\HHH$ the bundle whose fibers
 are given by $\ker \beta \cap \ker \eta$.
 Then we have the following splittings:
\begin{equation}\label{cs-splitting}
TN = T\FF_1 \oplus T\FF_2 = \bR W \oplus T\HHH \oplus T\FF_2 ,
\end{equation}
where $T\FF_1=\bR W \oplus T\HHH$. Moreover the two form $d\beta$ (resp. $\eta$) induces a symplectic form on $T\HHH$ (resp. $T\FF_2$).

%%%%%%%%%%%%%%%%%%%%%%%%%%%%%%%%%%%%%%%%%%%%%%%%%%%

\subsection*{The Boothby--Wang construction}\label{s:BW}

The Boothby-Wang fibration \cite{BW},
associates regular contact forms to integral symplectic forms. If
$(M,\omega)$ is a closed symplectic manifold and $\omega$
represents an integral class in $H^{2}(M;\bR)$ then there exists a
principal $\bS^{1}$-bundle $\pi\colon E\rightarrow M$ with Euler
class $[\omega]$ and a connection $1$-form $\alpha$ on it with
curvature $\omega$, i.e.~we have $d\alpha=\pi^{*}\omega$. As
$\omega$ is assumed to be symplectic on $M$, it follows that
$\alpha$ is a contact form on the total space $E$.

If $\omega$ is an arbitrary closed $2$-form representing an
integral cohomology class, we can again find a connection $1$-form
$\alpha$ with curvature $\omega$. If $\omega$ has constant rank
$2k$, then $\alpha$ has constant class $2k+1$, that is
$\alpha\wedge (d\alpha)^{k}\neq 0$, and $(d\alpha)^{k+1}= 0$.

This yields the following results from \cite{BK}:
\begin{theorem}[\cite{BK}]\label{thBWsp}
    Let $M$ be a closed manifold with a symplectic pair
    $(\omega_{1},\omega_{2})$. If $[\omega_{1}]\in H^{2}(M;\bR)$ is an
    integral cohomology class, then the total space of the circle
    bundle $\pi\colon E\rightarrow M$, with Euler class $[\omega_{1}]$,
    carries a natural $\bS^1$-invariant contact-symplectic pair.
    \end{theorem}

\begin{theorem}[\cite{BK}]\label{thBWcsp}
    Let $M$ be a closed manifold with a contact-symplectic pair
    $(\alpha,\beta)$. If $[\beta]\in H^{2}(M;\bR)$ is an integral
    cohomology class, then the total space of the circle bundle
    $\pi\colon E\rightarrow M$, with Euler class $[\beta]$,
    carries a natural $\bS^1$-invariant contact pair.
    \end{theorem}

\begin{corollary}[\cite{BK}]\label{corBWsp}
    If a closed manifold $M$ has a symplectic pair
    $(\omega_{1},\omega_{2})$ such that both $[\omega_{i}]\in
    H^{2}(M;\bR)$ are integral, then the fiber product of the two
    circle bundles with Euler classes equal to $[\omega_{1}]$ and
    $[\omega_{2}]$ respectively carries a natural $\bS^1$-invariant contact pair.
    \end{corollary}

In particular the Corollary \ref{corBWsp} affirms that starting
from a symplectic pair whose two forms represent integral classes,
then performing a double Boothby-Wang fibration, one obtains a
contact pair on the top.

\subsection{Almost contact structures}
An almost contact structure on a manifold $M$ is a triple $(\alpha, Z, \phi)$ of a
one-form $\alpha$, a vector field $Z$ and a field of endomorphisms $\phi$ of the tangent bundle of $M$,
 such that $\phi ^2 =-Id + \alpha \otimes Z$, $\phi(Z)=0$ and $\alpha(Z)=1$. In particular, it follows that $\alpha \circ \phi=0$ and that
 the rank of $\phi$ is $\dim M -1$.

If a manifold $M$ carries such a structure, one can consider an almost complex structure on $M \times \bR$. Every $Y \in \Gamma (T(M \times\bR))$
 can be written as $X + f \frac{d}{dt}$ for $X$ tangent to $M$ and
$f \in C^\infty (M \times \bR)$. Then the almost complex structure is defined as follows:
$$
J(X + f \frac{d}{dt})= \phi X -f Z + \alpha (X) \frac{d}{dt} .
$$

The almost contact structure is said to be normal if $J$ is
integrable. The integrability condition for $J$ is equivalent to
the following condition:
\begin{equation}
[\phi, \phi](X,Y)+2 d\alpha (X,Y) Z=0 , \, \forall X, Y \in \Gamma (TM) ,
%& L_ Z \phi=0 \\
%& L_Z \alpha=0 \\
%& [L_{\phi X}\alpha] (Y) -[L_{\phi Y}\alpha] (X)=0 .
\end{equation}
%It is a standard fact that the first equation implies the
%remaining ones.
where $[\phi, \phi]$ is the Nijenhuis tensor of $\phi$.\\
If $\alpha$ is a contact form and $(\alpha, Z, \phi)$ an almost
contact structure, we often refer to it as a \textit{contact form with structure
tensor} $\phi$. When the structure is normal we call it
\textit{normal contact form} for short.

\subsection{Contact-symplectic pairs as almost contact structures}\label{subsec:alm-cont-sympl-str}
A contact-symplectic pair $(\beta , \eta)$ on a manifold $N$ can be viewed
as a special almost contact structure (in \cite{Blair} D. Blair
considered similar structures) when it is endowed with an endomorphism $\psi$ of $TN$, satisfying
\begin{equation}\label{tensopsi}
\psi^2 = -Id + \beta \otimes W \, ,
\end{equation}
where $W$ is the Reeb vector field of $(\beta , \eta)$. Such a $\psi$ always exists
because on the kernel of $\beta$ the $2$-form $d\beta+\eta$ is
symplectic. By a standard polarization process, one can always
construct such a $\psi$ and an associated metric $g$, that is a metric satisfying the following conditions:
$$
g(X, \psi Y)= (d\beta + \eta)(X,Y) \, \text{and} \, g(X, W)= \beta (X) , \, \forall X, Y \in \Gamma (TN) .
$$
Since the symplectic subbundle determined by the kernel of $\beta$ can be split into two symplectic
subbundles $T\HHH$ and $T\FF_2$ as in \eqref{cs-splitting}, by polarization on both of them one can always construct a so called \textit{decomposable}
endomorphism $\psi$ which preserves the tangent spaces of the foliations (or equivalently $\psi (T\HHH )= T\HHH$ and $\psi (T\FF_2)= T\FF_2$) and an
associated metric $g$ for which the foliations are orthogonal with respect to $g$. We do not give the details for that, since we have proven the analog
of this statement for contact pair structures in \cite{BH2}.
\begin{definition}
An \textit{almost contact-symplectic structure} on a manifold $M$ is a triple $(\beta, \eta , \psi)$, where $(\beta, \eta)$ is a
contact-symplectic pair with Reeb vector filed $W$ and $\psi$ is an endomorphism of $TM$ satisfying \eqref{tensopsi}.
\end{definition}

\subsection{Contact pair structures}

This notion has been considered in \cite{BH2}. We recall here the definition and some basic properties which are useful in the sequel.
\begin{definition}[\cite{BH2}]
A \emph{contact pair structure} on a manifold $M$ is a triple
$(\alpha_1 , \alpha_2 , \phi)$, where $(\alpha_1 , \alpha_2)$ is a
contact pair and $\phi$ a tensor field of type
$(1,1)$ such that:
\begin{equation}\label{d:cpstructure}
\phi^2=-Id + \alpha_1 \otimes Z_1 + \alpha_2 \otimes Z_2 \, \, \text{and} \, \, \phi(Z_1)=\phi(Z_2)=0
\end{equation}
where $Z_1$ and $Z_2$ are the Reeb vector fields of $(\alpha_1 , \alpha_2)$.
\end{definition}
Moreover we have $\alpha_i \circ \phi =0$, $i=1,2$ and the rank of $\phi$ is equal to $\dim M -2$.
Recall that on a manifold $M$ endowed with a contact pair, there always exists an endomorphisms $\phi$ verifying \eqref{d:cpstructure}. Moreover, $\phi$
 can be chosen to be decomposable (\cite{BH2}, Proposition 5), that is:

\begin{definition}[\cite{BH2}]
The endomorphism $\phi$ is said to be \textit{decomposable} if
$\phi (T\mathcal{F}_i) \subset T\mathcal{F}_i$, for $i=1,2$.
\end{definition}
The condition for $\phi$ to be decomposable is equivalent to $\phi (T\mathcal{G}_i)= T\mathcal{G}_i$, $i=1,2$.

The following results are concerned with the structures induced on the leaves of the characteristic foliations:
\begin{proposition}[\cite{BH2}]\label{prop:induced C structures}
If $\phi$ is decomposable, then $(\alpha_1 , Z_1 ,\phi)$ (resp. $(\alpha_2 , Z_2 ,\phi)$) induces a contact form with structure tensor the restriction of
$\phi$ on the leaves of $\mathcal{F}_2$ (resp.$\mathcal{F}_1$).
\end{proposition}

\section{Almost complex structures}\label{s4}
To define a normal almost contact structure on a manifold $M$, one needs to
consider an almost complex structure on $M \times\bR$. In the case of a contact pair structure the almost
complex structure can be defined in a more natural and intrinsic way on the manifold.

\begin{definition}
Let $(\alpha_1 , \alpha_2 , \phi)$ be a contact pair
structure on a manifold $M$ and $Z_1$, $Z_2$ the Reeb vector fields of the
pair. The almost complex structure on $M$
\begin{equation}\label{almostcomplexJ}
J=\phi  - \alpha_2 \otimes Z_1 + \alpha_1 \otimes Z_2 ,
\end{equation}
is called the \textit{almost complex structure
associated} to $(\alpha_1 , \alpha_2 , \phi)$.
\end{definition}
%\begin{remark}
%A contact pair on a manifold gives rise to a non closed $2$-form
%$\Omega= \alpha_1 \wedge \alpha_2 -d\alpha_1 - d\alpha_2$. If the
%manifold is endowed with an associated metric $g$, then $g$ and $\Omega$ are the real and the imaginary part respectively, of
%an almost Hermitian metric, with respect to the almost complex structure $ J$. If $J$ is integrable, then $g+i \Omega$ is Hermitian.
%\end{remark}
We can also consider a second almost complex structure
\begin{equation}\label{almostcomplexT}
T=\phi  + \alpha_2 \otimes Z_1 - \alpha_1 \otimes Z_2 ,
\end{equation}
which is nothing but the almost complex structure associated to
the contact pair $(\alpha_2, \alpha_1, \phi)$ and commutes with
$J$.
\begin{remark}
The almost complex structure induced by $T$ on $T\mathcal G_1 \oplus T\mathcal G_2 $ is the same as $J$,
 but opposite to it on the subbundle $\bR Z_1 \oplus \bR Z_2$. Then the orientations induced by
$J$ and $T$ are opposite. In general one can not expect that both structures are integrable since this imposes some topological
obstructions, in particular on a four dimensional closed manifold (see \cite{Ko}).
\end{remark}

Recalling that the forms $\alpha_1, \alpha_2$ are invariant by the
Reeb vector fields, a straightforward calculation shows that the
Nijenhuis tensor of the almost complex structure $J$ associated to the
contact pair structure $(\alpha_1, \alpha_2, \phi)$ is given by:
\begin{equation}
\begin{split}
\label{intJ}
N_J (X,Y)= &[\phi, \phi](X,Y) +2 d\alpha_1 (X,Y) Z_1 +2
d\alpha_2 (X,Y) Z_2+\alpha_1 (X) [L_{Z_2} \phi] (Y)\\
& - \alpha_1 (Y) [L_{Z_2} \phi]
(X)+ \alpha_2 (Y) [L_{Z_1} \phi] (X)-\alpha_2 (X) [L_{Z_1} \phi
](Y)\\
&+ [(L_{\phi X} \alpha_1) (Y)-(L_{\phi Y} \alpha_1) (X)] Z_2
+[(L_{\phi Y} \alpha_2) (X)-(L_{\phi X} \alpha_2) (Y)] Z_1 \, ,
\end{split}
\end{equation}
for each $X , Y \in \Gamma (TM)$, where $L_X$ is the Lie derivative along $X$, $[\phi, \phi]$ is
the Nijenhuis tensor of $\phi$.

The Nijenhuis tensor $N_T$ of the almost complex structure defined in \eqref{almostcomplexT}, is obtained from $N_J$
by interchanging the role of the forms $\alpha_1$, $\alpha_2$ and their Reeb vector fields. Then, for each $X , Y \in \Gamma (TM)$ we have:
\begin{equation}
\begin{split}
\label{intT}N_T (X,Y)=& [\phi, \phi](X,Y) +2 d\alpha_1 (X,Y) Z_1 +2
d\alpha_2 (X,Y) Z_2-\alpha_1 (X) [L_{Z_2} \phi] (Y)\\
& + \alpha_1 (Y) [L_{Z_2} \phi]
(X)-\alpha_2 (Y) [L_{Z_1} \phi] (X)+\alpha_2 (X) [L_{Z_1} \phi
](Y)\\
&-[(L_{\phi X} \alpha_1) (Y)-(L_{\phi Y} \alpha_1) (X)] Z_2
-[(L_{\phi Y} \alpha_2) (X)-(L_{\phi X} \alpha_2) (Y)] Z_1 \, .
\end{split}
\end{equation}

The vanishing of both $N_J$ and $N_T$ is equivalent to the
vanishing of their sum and their difference. Since $[L_{Z_i}\phi]
(X)$ is in the kernel of $\alpha_1$ and $\alpha_2$ for every $X \in \Gamma (TM)$,
the integrability of both $J$ and $T$ is equivalent to the
following system:
\begin{equation}
\left\{
\begin{aligned}
&[\phi , \phi ](X,Y)+2 d\alpha_1 (X,Y) Z_1 +2 d\alpha_2 (X,Y) Z_2=0 \label{eq1}\\
&-\alpha_1 (X) [L_{Z_2} \phi] (Y) + \alpha_1 (Y) [L_{Z_2} \phi]
(X)-\alpha_2 (Y) [L_{Z_1} \phi] (X)+\alpha_2 (X) [L_{Z_1} \phi ](Y) =0\\
&[(L_{\phi X} \alpha_1) (Y)-(L_{\phi Y} \alpha_1) (X)] Z_2=0\\
&[(L_{\phi Y} \alpha_2) (X)-(L_{\phi X} \alpha_2) (Y)] Z_1 =0 ,
\end{aligned}
\right.
\end{equation}
for every  $X , Y \in \Gamma (TM)$.
Now, putting $Y=Z_i$ in the first equation, one obtains
$L_{Z_i}\phi=0$, which implies the second equation. Applying
$\alpha_i$ to $N_J(\phi X, Y)$ gives the last equations.

These observations yield the following theorem:
\begin{theorem}\label{intJT}
The integrability of both $J$ and $T$ is equivalent to the
following equation:
\begin{equation}\label{eq:intJT}
[\phi, \phi](X,Y) +2 d\alpha_1 (X,Y) Z_1 +2 d\alpha_2 (X,Y) Z_2=0
\, \, \forall X , Y \in \Gamma(TM) \, .
\end{equation}
\end{theorem}

By using the splitting \eqref{split-CP1}, the equation \eqref{eq:intJT} is equivalent to the following system:
\begin{align}
\label{intJTbis1} &[\phi, \phi](X,Y) +2 d\alpha_1 (X,Y) Z_1=0 \, \,
\forall X , Y \in \Gamma(T\mathcal{F}_2) \, \, ,\\
\label{intJTbis2} &[\phi, \phi](X,Y) +2 d\alpha_2 (X,Y)
Z_2=0 \, \, \forall X , Y \in \Gamma(T\mathcal{F}_1) \, \, , \\
\label{intJTbis3} & [\phi, \phi](X,Y)=0 \, \, \forall X \in
\Gamma(T\mathcal{F}_1) \, , \, \forall Y \in \Gamma(T\mathcal{F}_2)  .
\end{align}

In analogy with the case of the almost contact structures we give
the following definition:
\begin{definition}
A contact pair structure $(\alpha_1, \alpha_2, \phi)$ on a
manifold $M$ is said to be a \textit{normal contact pair} if the
Nijenhuis tensors $N_J$ and $N_T$ vanish identically.
\end{definition}
The equation \eqref{eq:intJT} states exactly the normality of the contact pair structure.

%\begin{eqnarray}\left\{
%&[\phi, \phi](X,Y) +2 d\alpha_1 (X,Y) Z_1=0 \, \,
%\forall X , Y \in \Gamma(T\mathcal{F}_2) \, \, ,\\
%&[\phi, \phi](X,Y) +2 d\alpha_2 (X,Y)
%Z_2=0 \, \, \forall X , Y \in \Gamma(T\mathcal{F}_1) \, \, , \\
%& [\phi, \phi](X,Y)=0 \, \, \forall X \in
%\Gamma(T\mathcal{F}_1) \, , \, \forall Y \in \Gamma(T\mathcal{F}_2) . \right .
%\end{eqnarray}

\subsection{ Decomposable $\phi$ and induced contact structures}
In this case we already remarked that the contact pair structure
induces contact forms with structure tensor $\phi$, on the leaves
of the characteristic foliations $\FF_1$ and $\FF_2$ (Proposition \ref{prop:induced C structures}). Applying Theorem \ref{intJT}, we
have:

\begin{corollary}\label{corinducedstr1}
Let $(\alpha_1, \alpha_2, \phi)$ be a contact pair with decomposable $\phi$. The structure is normal if and only if the induced structures are normal and
\eqref{intJTbis3} is satisfied.
\end{corollary}
\begin{proof}
When $\phi$ is decomposable, \eqref{intJTbis1} and \eqref{intJTbis2} are equivalent to the normality of the induced structures.
\end{proof}
A partial converse of this corollary is the following:
\begin{corollary}\label{corinducedstr2}
If $\phi$ is decomposable and both characteristic foliations are
normal for the induced structures, then $J$ is integrable if and
only if $T$ is integrable.
\end{corollary}
\begin{proof}
Let us suppose that $J$ is integrable. We want to prove that \eqref{intJTbis1}, \eqref{intJTbis2} and \eqref{intJTbis3} are
satisfied. The first two equations are a consequence of the normality
of the induced structures. Moreover, this implies
\begin{eqnarray}
\label{eq:1:intJT}&[L_{Z_1}\phi] (X) =0 \, \, \forall X \in \Gamma(T\mathcal{F}_2) \; ,\\
\label{eq:2:intJT}&[L_{Z_2}\phi] (X) =0 \, \, \forall X \in \Gamma(T\mathcal{F}_1) \, .
\end{eqnarray}
Because $N_J$ vanishes, for $i=1,2$ we have $N_J(X, Z_i)=0$ for every
$X$. Combining this with \eqref{eq:1:intJT} and \eqref{eq:2:intJT}, we obtain
$L_{Z_i}\phi=0$. This implies that for $X,Y$ tangent to different
foliations
$$
0=N_J (X, Y)=[\phi, \phi](X,Y),
$$
which gives
\eqref{intJTbis3}. We argue similarly for $T$ and this completes the
proof.
\end{proof}
An immediate consequence is the Theorem of Morimoto  for a product of
contact manifolds (see \cite{Mori}). If $J$ and $T$ are the almost complex structures defined in \eqref{almostcomplexJ} and \eqref{almostcomplexT} respectively, then we have:
\begin{corollary}[\cite{Mori}]\label{corolmorimoto}
Suppose that $(M_1, \alpha_1 , \phi_1)$ and $(M_2,
\alpha_2 , \phi_2)$ are contact manifolds with structure tensor $\phi_1$ and $\phi_2$ respectively. Then the contact
pair structure $(\alpha_1, \alpha_2, \phi_1 \oplus \phi_2)$ on
$M_1 \times M_2$ is normal if and only if $( \alpha_1 , \phi_1)$ and $(
\alpha_2 , \phi_2)$ are normal as almost contact structures.
\end{corollary}
\begin{proof}
It is clear that in this case $\phi$ is decomposable. If the
almost contact structures on $M_1$ and $M_2$ are normal, then
\eqref{intJTbis1} and \eqref{intJTbis2} are verified. Equation
\eqref{intJTbis3} is automatically satisfied if $X$ and $Y$ are
tangent to different foliations because the manifold is a product
and the vector fields can be supposed to commute. The converse is
true by Corollary \ref{corinducedstr1}.
\end{proof}
We give now an example of a manifold endowed with a normal contact pair, with decomposable
$\phi$ and where the induced structures are normal, but
the manifold is not itself a product of two contact manifolds:
\begin{example}\label{exD1}
Let $M=\widetilde{SL_2}$ be the universal covering of the identity
component of the isometry group of the hyperbolic plane $\bH^2$
endowed with an invariant normal contact form $\alpha$ (see
\cite{Geiges2}) and $N= M \times M$. It is well known that $N$ admits cocompact irreducible lattices $\Gamma$ (see
\cite{Bor}). This means that $ \Gamma$ does not admit any subgroup
of finite index which is a product of two lattices of $M$. The
manifold $N$ can be endowed with the obvious contact pair
structure and by the invariance of the contact forms by $\Gamma$,
the contact pair descends to the quotient and is normal. Even if
the local structure is like a product, globally the foliations can
be very interesting in the sense that both could have dense
leaves.
\end{example}

Now we want to investigate deeply the condition
$L_{Z_i}\phi=0$, for $i=1,2$, since this condition is the analog of the
$K$-contact condition for the almost contact structures. In the
proof of Corollary \ref{corinducedstr2} we saw that, if the induced structures are normal, the condition $L_{Z_i}\phi=0$ is
 necessary to the integrability of both almost complex
structures . One can ask if this condition together with the integrability of one of the almost complex
structures is weaker than the integrability of both of them. We
begin with the following proposition:
\begin{proposition}\label{JTinth0}
Let $M$ be a manifold endowed with a contact pair structure
$(\alpha_1, \alpha_2, \phi)$ together with a decomposable $\phi$
and suppose that the almost complex structure $J$ associated to
the pair is integrable. Let $T$ be the almost complex structure
associated to $(\alpha_2, \alpha_1, \phi)$. Then the following
properties are equivalent:
\begin{enumerate}
\item $T$ is integrable;
\item $L_{Z_1}\phi=0$;
\item $L_{Z_2}\phi=0$.
\end{enumerate}
\end{proposition}
\begin{proof}
Suppose that both almost complex structures are integrable, then
we have already seen in the proof of Theorem \ref{intJT} that this
implies $L_{Z_i}\phi =0$, $i=1,2$.

Conversely, since $J$ is integrable, for every $X \in \Gamma (TM)$ we have
$$
0=N_J(X,Z_2)=\phi([L_{Z_1}\phi](X))-[L_{Z_2}\phi](X) \, ,
$$
which implies that $[L_{Z_2}\phi](X)=0$ if and only if
$[L_{Z_1}\phi](X)=0$. It remains to show that $T$ is also
integrable. This can be easily seen by calculating its Nijenhuis tensor $N_T(X,Y)$. One
has just to remark that when $X,Y$ are tangent to the same foliation,
since $\phi$ is decomposable and $Z_1 , Z_2$ are not in $\ker
\alpha_1 \cap \ker \alpha_2$, then the equations obtained are
exactly \eqref{intJTbis1} and \eqref{intJTbis2}. Again,
by the decomposability of $\phi$, if $X$ and $Y$ are tangent to different foliations,
one obtains \eqref{intJTbis3}.
\end{proof}
Combining Theorem \ref{intJT} and Proposition \ref{JTinth0} we
obtain the following theorem:
\begin{theorem}\label{fifi}
Let $(\alpha_1, \alpha_2, \phi)$ be a contact pair structure on a manifold $M$ with
a decomposable $\phi$ and such that $L_{Z_1}\phi=0$ (resp. $L_{Z_2}\phi=0$),
then the following conditions are equivalent:
\begin{enumerate}
\item[i)] $J$ is integrable;

\item[ii)] $T$ is integrable;

\item[iii)] the induced structures are normal and $[\phi, \phi](X,Y)=0$
for all $X \in \Gamma(T\FF_1)$ and for all $Y \in \Gamma(T\FF_2)$.
\end{enumerate}
Moreover these equivalent conditions imply $L_{Z_2}\phi=0$ (resp. $L_{Z_1}\phi=0$).
\end{theorem}

\subsection{Non Morimoto case}
In general when $\phi$ is decomposable, if the induced structures
are normal, the condition $L_{Z_i}\phi =0$ for $i=1,2$ does not imply
the normality of the whole structure. The following examples show
that the situation in the general case can be more complicated. In
particular, they show that there are contact pair structures with
decomposable $\phi$ and normal induced structures but, unlike the
Morimoto construction, the contact pair structure is not normal. There neither $J$ nor $T$ is integrable and \eqref{intJTbis3} is not satisfied.
\begin{example}
Consider the simply connected Lie group $G$ with structure
equations:
\begin{eqnarray*}
&d\omega_1= d\omega_6=0 \; \; , \; \; d\omega_2= \omega_5 \wedge
\omega _6\\
&d\omega_3=\omega_1 \wedge \omega_4 \; \; , \; \;
d\omega_4= \omega_1 \wedge \omega_5 \; \; , \; \; d\omega_5 =
\omega_1 \wedge \omega_6 ,
\end{eqnarray*}
where the $\omega_i$'s form a basis for the cotangent space of $G$
at the identity.

The pair $(\omega_2 , \omega_3)$ is a contact pair of type $(1,1)$
with Reeb vector fields $(X_2, X_3)$, the $X_i$'s being dual to
the $\omega_i$'s. Now define $\phi$ to be zero on the Reeb vector
fields and
$$
\phi(X_5)=X_6 \; \; , \; \; \phi(X_6)=-X_5 \; \; , \; \; \phi
(X_1)= X_4 \; \; , \; \; \phi (X_4)=-X_1 \, .
$$
Since the kernel of $\omega_2 \wedge d\omega_2$ is generated by $X_1 , X_3, X_4$, it is clearly preserved by $\phi$. The same holds for the kernel of
 $\omega_3 \wedge d\omega_3$. Moreover $\phi$ is easy verified to be invariant
under the flows of the Reeb vector fields. The induced
structures are normal, but not the whole structure because it is
well known that this Lie algebra does not admit any complex
structure.

Since the structure constants of the group are rational, there
exist lattices $\Gamma$ such that $G/ \Gamma$ is compact and then
we obtain nilmanifolds carrying the same type of structure.
\end{example}

\begin{example}
The Lie group having the following structure equations admits
invariant complex structures (see \cite{Simon}):
$$
d\omega_1=d\omega_2=d\omega_3=0 \; \; , \; \;  d\omega_4= \omega_1 \wedge
\omega_2 \; \; , \; \; d\omega_5= \omega_1 \wedge \omega_3 \; \; , \; \;
d\omega_6= \omega_2 \wedge \omega_4 \, \, .
$$
The pair $(\omega_5 , \omega_6)$ is a contact pair of type
$(1,1)$. A straightforward calculation shows that every invariant
contact pair structure of type $(1,1)$ with invariant and
decomposable $\phi$ has normal induced structures but the whole
structure is not normal since \eqref{intJTbis3} is not
satisfied.
\end{example}
According to the result of Morimoto (Corollary \ref{corolmorimoto}), the manifolds carrying contact pair structures
in the previous examples can not be, even locally, products of manifolds endowed with normal
contact forms.

%%%%%%%%%%%%%%%%
%%%%%%%%%%%%%%%%
%%%%%%%%%%%%%%%%%%%%%%%%%%%%%%%%%%%%%
%%%% vanishing using the splitting
%%%%%%%%%%%%%%%%%%%%%%%%%%%%%%%%%%
%For the sequel it is useful to remark that by the splitting
%$TM=T\mathcal{F}_1+T\mathcal{F}_2$, the vanishing condition for
%$N_J$ is equivalent to the following three equations
%\begin{eqnarray*}
%&[\phi, \phi](X,Y) +2d\alpha_2 (X,Y) Z_2+\alpha_2 (Y)
%[L_{Z_1}\phi] (X)-\alpha_2 (X) [L_{Z_1} \phi ](Y)+\\
%&[(L_{\phi Y} \alpha_2) (X)-(L_{\phi X} \alpha_2) (Y)] Z_1=0 \; \;
%\forall X , Y \in \mathcal{F}_1 \; ,
%\end{eqnarray*}
%
%\begin{eqnarray*}
%&[\phi, \phi](X,Y) +2d\alpha_1 (X,Y) Z_1+\alpha_1 (X) [L_{Z_2}
%\phi] (Y) - \alpha_1 (Y) [L_{Z_2} \phi] (X)+\\
%&[(L_{\phi X} \alpha_1) (Y)-(L_{\phi Y} \alpha_1) (X)] Z_2=0 \; \;
%\forall X , Y \in \mathcal{F}_2 \; ,
%\end{eqnarray*}
%
%\begin{eqnarray*}
%&[\phi, \phi](X,Y)+[(L_{\phi X} \alpha_1) (Y)] Z_2+[(L_{\phi Y}
%\alpha_2) (X)] Z_1=0 \; , \forall X  \in \mathcal{F}_1, \forall Y
%\in \mathcal{F}_2 \; .
%\end{eqnarray*}
%%%%%%%%%%%%%%%%%%%%%%%%%%%%%%%
%%%%%%%%%%%%%%%%%%%%%%%%%%%%%%%%%
%%%%%%%%%%%%%%%%%%%%%%%%%%%%%%%%%
\subsection{Contact pairs of type $(h,0)$}
In the particular case of a manifold $M$, endowed with a contact
pair structure $(\alpha_1 , \alpha_2, \phi)$ of type $(h,0)$, the
$1$-form $\alpha_2$ is closed and the Nijenhuis tensors of the
almost complex structures $J$ and $T$ associated to the pair
simplify further.
%\begin{alignat*}{1}
%N_J(X,Y)=&[\phi, \phi](X,Y) +2 d\alpha_1 (X,Y) Z_1 + \alpha_1 (X) [L_{Z_2} \phi] (Y) - \alpha_1 (Y) [L_{Z_2} \phi]
%(X)+\\
%+& \alpha_2 (Y) [L_{Z_1} \phi] (X)-\alpha_2 (X) [L_{Z_1} \phi
%](Y)+ [(L_{\phi X} \alpha_1) (Y)-(L_{\phi Y} \alpha_1) (X)] Z_2 .
%\end{alignat*}
Moreover the tensor $\phi$ is automatically decomposable because
$\alpha_2 \circ \phi=0$ implies that $\phi (T\mathcal{F}_2)
\subset T\mathcal{F}_2$. Since $\Gamma (T\mathcal{F}_1) $ is generated by $Z_2$ and $\phi(Z_2)=0$, we also have
$\phi (T\FF_1) \subset T\mathcal {F}_1$.

The following is a variation of the Theorem \ref{fifi} for contact
pairs of type $(h,0)$:
\begin{theorem}\label{moricph0}
Let $M$ be a manifold endowed with a contact pair structure $(\alpha_1 , \alpha_2 , \phi)$  of
type $(h,0)$, such that $L_{Z_2} \phi=0$. Then  $(\alpha_1 , \alpha_2 , \phi)$ is a normal contact pair if and only if
$(\alpha_1 , \phi)$ induced on every leaf of $\mathcal{F}_2$
is normal.
\end{theorem}
%\begin{proof}
%If the whole structure is normal, Theorem \ref{fifi} implies that
%the induced structure are normal.
%
%Conversely, let us suppose that every induced structure is normal.
%To prove the va\-nishing of $N_J$ we have to show:
%$$
%N_J (X,Y)=0 \; \forall X,Y \in \Gamma(T\mathcal {F}_2) \, \, \; ,
%N_J (X,Z_2)=0 \; \forall X \in \Gamma(T\mathcal {F}_2) .
%$$
%The first condition is clearly true by the normality of the
%induced structures and the second one is true since
%$L_{Z_2}\phi=0$.
%\end{proof}
Defining normality for a contact manifold $(M, \alpha)$ with
structure tensor $\phi$, is the same as considering the contact
pair $(\alpha, dt)$ on $M \times \bR$ and asking for its almost
complex structure to be integrable. This is exactly the local
situation of the previous theorem.
\begin{remark}
A manifold endowed with a normal contact pair of type $(h,0)$ can
be viewed as an even analog of a cosymplectic manifold.
\end{remark}
%%%%%%%%%%%%%%%%%%%%%%%%%%%%%%%%%%%
%%%%%%%%%%%%%%%%%%%%%%%%%%%%%%%%%%%
%%%%%examples
We end this section with the following example:
\begin{example}\label{exD2}
Let us consider the simply connected nilpotent Lie group $Nil^4$,
having the following structure equations:
$$
d\omega_1=d\omega_4=0 \; \; , \; \; d\omega_2=\omega_1 \wedge
\omega_4 \; \; , \; \; d\omega_3=\omega_2 \wedge \omega_4 \, .
$$
The pair $(\omega_ 3,\omega_ 1 )$ is a contact pair of type
$(1,0)$. Since the structure constants of the group are rational,
then there exist cocompact lattices and the corresponding
nilmanifold are endowed with a contact pair structure and hence
with an almost complex structure. Nevertheless this contact pair
can not be normal since no such nilmanifold admits complex
structures. This can be seen for
example by saying that such a nilmanifold has first Betti number
$b_1=2$ (see \cite{N}) and if it is complex with even first Betti
number then it must be K\"ahler by \cite{buch}. But the only
nilmanifolds which are K\"ahler must be Tori (see \cite{BG}) and
this is not the case.
\end{example}

\subsection{Remarks on bicontact Hermitian manifolds}
Contact pairs appeared firstly in \cite{Blair2}, where they arose in the context of the Hermitian geometry with the name \textit{bicontact}. 

More precisely a \textit{bicontact Hermitian} manifold is a Hermitian manifold $(M, J, g)$ together with a unit vector field $U$ such that $U$ and $V=JU$ are infinitesimal automorphisms of the Hermitian structure. Let $u$ and $v$ be the covariant forms of $U$ and $V$ respectively. The bicontact manifold $M$ is said to be of bidegree $(1,1)$ if $du$ is of bidegree $(1,1)$ and in this case $dv$ is of bidegree $(1,1)$ too. 

Actually, a bicontact Hermitian manifold $(M, J, g, U)$ of bidegree $(1,1)$ can be regarded as a manifold endowed with a normal contact pair structure $(u, v, \phi)$, where $\phi=J+ v \otimes U - u \otimes V$ is decomposable, together with a metric $g$ which is compatible in the sense of \cite{BH2}. This easily follows from Propositions $2.7$ and $2.8$ of \cite{Blair2} and by the fact that the bidegree $(1,1)$ of $du$ implies the decomposability of $\phi$.
By using Propositions \ref{JTinth0} and \ref{prop:induced C structures} and the local model for a contact pair (see \cite{Bande1, BH}), Theorem $4.4$ of \cite{Blair2} can be restated in terms of  normal contact pairs with decomposable $\phi$. Moreover Theorem $4.4$ of \cite{Blair2} implies the necessary condition of Corollary \ref{corinducedstr1}.

\section{Constructions on flat bundles}\label{s5}

Flat bundles are fibre bundles with a foliation transverse and
complementary to the fibre and have been useful to
construct symplectic pairs in \cite{BK}. In the same paper
was pointed out that one can use these bundles to construct
contact pairs. We describe the general construction of flat
bundles and then we specialize to contact pair structures.

Let $B$ and $F$ be two connected manifolds, and let
$\rho\colon\pi_{1}(B)\rightarrow\Diff(F)$ be a representation of
the fundamental group of $B$ in the group of diffeomorphisms of
$F$. The suspension of $\rho$ defines a horizontal foliation (whose holonomy is $\rho$) on
the fiber bundle $\pi\colon M_{\rho}\rightarrow B$ with fiber $F$ and
total space
$$M_{\rho}=(\tilde
B\times F)/\pi_{1}(B),$$ where $\pi_{1}(B)$ acts on the universal
covering $\tilde B$ by covering transformations and on $F$ via
$\rho$. We have the following commutative diagram:
$$
\begin{CD}
\tilde B \times F @>\pi_{\rho}>> (\tilde
B\times F)/\pi_{1}(B)\\
@VQ VV @V\pi VV\\
\tilde B @>p>> B
\end{CD}
$$
where $Q$ is the projection on the first factor, $p$ the covering
projection, $\pi$ the projection of the bundle and $\pi_{\rho}$
the quotient map.

Let us consider contact manifolds $(B, \alpha_1, Z_1 ,\phi_1)$
and $(F, \alpha_2 , Z_2 ,\phi_2)$  with structure tensors $\phi_1
$ and $\phi_2$ respectively. Instead of taking a representation of
$\pi_1 (B)$ in $\Diff(F)$, we take a representation $\rho$ in $\Cont
(F , \phi_2)$, the group of contactomorphisms preserving $\phi_2$,
and we construct the flat bundles by using this representation.

Let $(\alpha_1 ,\alpha_2, \phi_1\oplus\phi_2)$ be the contact pair structure of $B \times F$, $J$ its almost complex structure and $T$ the
almost complex structure of $(\alpha_2, \alpha_1 ,\phi_1\oplus\phi_2)$.
Then, by Morimoto's result (Corollary \ref{corolmorimoto}), $J$ is integrable if and only if $(B, \alpha_1, Z_1 ,\phi_1)$ and
$(F, \alpha_2 , Z_2 ,\phi_2)$ are normal and this if and only if $T$ is integrable.

The manifold $\tilde B \times F$
is naturally endowed with a contact pair structure $(\tilde \alpha_1 ,\alpha_2, \tilde\phi_1
\oplus\phi_2)$ where $\tilde \alpha_1$ and $\tilde \phi_1$ are the lift to $\tilde B$ of $\alpha_1$ and $\phi_1$ respectively.
The almost complex structure $\tilde J$ associated to it, is the lift of $J$ and then it is integrable if and only if $J$ is integrable.

Since $(\tilde\alpha_1 ,\alpha_2, \tilde \phi_1\oplus\phi_2)$ is invariant by the action of $\pi_{1}(B)$,
 the total space of the flat bundle $M_{\rho}=(\tilde B\times F)/\pi_{1}(B)$ is endowed
with a contact pair structure, denoted by $(\tilde \alpha_1 ,\alpha_2, \tilde \phi_1
\oplus\phi_2)_{\rho}$. The almost complex structure $\tilde J$ descends to the quotient and it defines the almost complex structure $J_{\rho}$ of
$(\tilde \alpha_1 ,\alpha_2, \tilde \phi_1
\oplus\phi_2)_{\rho}$. Then $J_{\rho}$ is integrable if and only if its lift $\tilde J$ is integrable.
Starting with $T$, we obtain the almost complex structure $T_{\rho}$ associated to $(\alpha_2, \tilde \alpha_1 , \tilde \phi_1
\oplus\phi_2)_{\rho}$.

The above discussion yields the following theorem:
\begin{theorem}\label{th:flat-bundles}
Let $(B, \alpha_1,Z_1 , \phi_1)$ and $(F, \alpha_2 , Z_2 ,
\phi_2)$ be two connected contact manifolds with structure tensors $\phi_1$
and $\phi_2$ respectively and $\rho$ any representation of $\pi_1
(B)$ in $\Cont (F , \phi_2)$. Then the flat bundle
$M_{\rho}=(\tilde B\times F)/\pi_{1}(B),$ is naturally endowed
with a contact pair structure $(\tilde \alpha_1 ,\alpha_2, \tilde
\phi_1 \oplus\phi_2)_{\rho}$. This contact pair structure is normal if and only if
$(B, \alpha_1, \phi_1)$ and $(F, \alpha_2 , \phi_2)$ are normal.
\end{theorem}

%An example for the Theorem \ref{th:flat-bundles} is given by taking a fibre $F$ endowed with a
%normal contact metric structure. In this case the tensor $\phi_2$
%and the contact form are invariant by the diffeomorphisms
%generated by the flow of Reeb vector field. Then every such a
%diffeomorphism is in $Cont (F , \phi_2)$. More generally one can
%consider a $K$-contact structure.

By choosing normal contact forms on $B$ and $F$ and a non trivial
representation $\rho$, this construction furnishes examples of
complex manifolds which are locally but not globally product of
contact manifolds as in Morimoto's theorem (see Corollary
\ref{corolmorimoto}). Here is an explicit example:

\begin{example}
Consider a closed manifold $F$ endowed with a normal $K$-contact
structure, for example Sasakian, $\left(  \alpha_{2},Z_{2},\phi_{2},g\right) $. In this case the one parameter group of
diffeomorphisms $\left\{ \varphi_{t},t\in\mathbb{R}\right\}$
generated by the flow of the Reeb vector field $Z_{2}$ is a non
trivial subgroup of $\Cont\left(  F, \phi_{2}\right)  .$ Pick any
element $\varphi_{a}$ which is not the identity. Let $B= Nil^{3}/
\Gamma$ where $Nil^{3}$ is the Heisenberg group of upper
triangular real $\left(  3\times3\right)  -$matrices
\[
\left(
\begin{array}
[c]{lll}%
1 & x & z\\
0 & 1 & y\\
0 & 0 & 1
\end{array}
\right)
\]
and $\Gamma$ the discrete subgroup consisting of all its elements with integer
entries. Then $B$ is a closed 3-manifold for which $\pi_{1}\left(  B\right)
=\Gamma$, since $Nil^3$ is simply connected. The invariant normal contact structure $\left(  \omega
,Z,\Phi\right)  $ on $Nil^{3}$ given by $\omega=dz-xdy$ and $\Phi$ defined by
$\Phi(\frac{\partial}{\partial x})=-\left(  \frac{\partial}{\partial y}+x \frac{\partial}{\partial z}\right)$
descends to $B$ as a normal contact structure $\left(  \alpha_{1},Z_{1}
,\phi_{1}\right)$ (see \cite{Geiges2}). Now choose the representation $\rho:\pi_{1}\left(
B\right)  \rightarrow \Cont\left(  F,\phi_{2}\right)$ defined  by
\[
\rho\left(
\begin{array}
[c]{lll}%
1 & p & r\\
0 & 1 & q\\
0 & 0 & 1
\end{array}
\right)  =\left(  \varphi_{a}\right)  ^{p},
\]
for all integers $p,$ $q$ and $r$.
In the same way, we can find other examples by using the Geiges'
classification \cite{Geiges2} and \cite{scott}.
\end{example}

\section{Constructions on Boothby--Wang fibrations}\label{s6}
In this section we use the Boothby--Wang fibration to construct
$\bS^1$-invariant contact pair structures on the total space of a principal circle
bundle over a base space endowed with a contact-symplectic pair
structure.

For a given closed manifold $B$ endowed with a contact-symplectic
pair $(\beta, \eta)$, if $\eta$ has integral cohomology class, as
showed in Section \ref{s:prelim}, one can construct a
Boothby--Wang fibration and obtain as total space a manifold $M$
endowed with a contact pair $(\alpha_1, \alpha_2)$ where
$\alpha_2=\pi^* (\beta)$ and $\alpha_1$ is the connection form of
the bundle and then $d\alpha_1=\pi^*(\eta)$, where $\pi$ is the bundle
projection. Let $Z_1$ and $Z_2$ be the Reeb vector fields of $(\alpha_1, \alpha_2)$. Then $Z_1$ is tangent to the action and $Z_2$ is the
horizontal lift of the Reeb vector field $W$ of $(\beta, \eta)$ with respect to the connection form $\alpha_1$.

The following result is the analog for contact pair structures of the construction used in \cite{Mori, Hata}:
\begin{theorem}\label{th:BW-cont-pair-str}
The total space $M$ of a Boothby--Wang fibration over a closed base space $B$
endowed with an almost contact-symplectic structure $(\beta, \eta, \psi)$, where $[\eta]
\in H^2(B, \bZ)$, is naturally
endowed with a $\bS^1$-invariant contact pair structure $(\alpha_1, \alpha_2, \phi)$. Moreover, if $\psi$ is decomposable so is $\phi$.
\end{theorem}
\begin{proof}
With the previous notations, let $(\alpha_1, \alpha_2)$ be the contact pair on $M$, with Reeb vector fields $Z_1$, $Z_2$. For any tangent vector $Y$ of
$B$ at $q= \pi (p)$, we denote by $Y_p ^*$ the horizontal lift (with respect to the connection form $\alpha_1$) of $Y$ at $p \in M$.
Let $\phi$ be the endomorphism of $TM$ defined as follows
$$
\phi _p X= (\psi \pi _* X)_p^* \, ,
$$
for every $X \in T_p M$, $\pi_*$ being the differential of the projection $\pi$.

The triple $(\alpha_1 , \alpha_2, \phi)$ is a contact pair
structure on $M$. To see that, we first remark that $\phi (X^*)= (\psi
X)^*$ and $(\pi_* X)^*=X-  \alpha_1(X) Z_1$. Then we have
$$
\phi^2 (X)= \phi (\psi \pi_* X)^*=(\psi^2 \pi_* X)^*=(-\pi_* X +
\beta(\pi_* X)W)^*=-X + \alpha_1 (X) Z_1 + \alpha_2 (X) Z_2 .
$$
Moreover, we have $\phi Z_1=0$, because $\pi_* Z_1=0$ and $\phi Z_2=\phi (W^*)=(\psi
W)^*=0$, because $\psi W=0$ by the definition of almost contact-symplectic structure. If $\psi$ is decomposable, the decomposability of $\phi$ can
be easily verified on lifted vector fields. Observe that $L_{Z_1} \phi =0$ by construction.
\end{proof}

Now we want to relate the normality of the contact pair structure
on the total space to that of the almost contact-symplectic structure on
the base. With the previous notations we have:
\begin{lemma}\label{lemma:BW-cont-pair-str-normality}
Let $B$ be a closed manifold endowed with an almost
contact-symplectic structure $(\beta, \eta, \psi)$ with $[\eta]
\in H^2(B, \bZ)$ and $M$ the total space of the corresponding
Boothby--Wang fibration, endowed with the $\bS^1$-invariant contact pair structure
$(\alpha_1 , \alpha_2, \phi)$ of Theorem
\ref{th:BW-cont-pair-str}. Then the almost complex structure $J$
associated to $(\alpha_1 , \alpha_2, \phi)$ is integrable if and
only if the following conditions on the base are satisfied:
\begin{align}
-2 \eta(\psi X, \psi Y)+2 \eta(X,Y)-d\beta (\psi X, Y)- d\beta
(X , \psi Y)&=0 \label{eq:1:BW}\\
[\psi, \psi] (X,Y) + 2 d\beta (X,Y) + \eta (\psi X, Y)+ \eta (X,
\psi Y)&=0 \label{eq:2:BW}\\
L_W \psi &=0 \label{eq:3:BW}.
\end{align}
\end{lemma}
\begin{proof}
The tensor $N_J$ vanishes if and only if $N_J(Z_1, X^*)=0$ and
$N_J( X^*, Y^*)=0$ for every lifted vector fields $X^*, Y^*$ (with respect to the connection form $\alpha_1$) and
for the vertical vector field $Z_1$. A straightforward calculation shows that $N_J(Z_1, X^*)=0$ is equivalent to \eqref{eq:3:BW} and
$N_J( X^*, Y^*)=0$ is equivalent to \eqref{eq:1:BW} and \eqref{eq:2:BW}.
\end{proof}

As a consequence of the above lemma we have:
\begin{theorem}\label{corBW}
With the same assumptions as  in Lemma
\ref{lemma:BW-cont-pair-str-normality}, if $\eta$ is invariant
under $\psi$, that is $\eta(\psi X, \psi Y)=\eta(X,Y)$, the $\bS^1$-invariant
contact pair structure on the total space of the Boothby--Wang
fibration has integrable $J$ if and only if the almost
contact-symplectic structure on the base is a normal almost
contact structure.
\end{theorem}
\begin{proof}
If $\eta$ is invariant under $\psi$, the conditions \eqref{eq:1:BW}, \eqref{eq:2:BW} and \eqref{eq:3:BW} reduce to
the following system
\begin{equation}
\left \{
\begin{aligned}
L_W \psi &=0\\
-d\beta (\psi X, Y)- d\beta(X , \psi Y)&=0\\
[\psi, \psi] (X,Y) + 2 d\beta (X,Y)&=0 .
\end{aligned} \right.
\end{equation}
The third equation implies the others and it is exactly the
condition for $(\beta, W, \psi)$ to be a normal almost contact
structure.
\end{proof}
\begin{theorem}\label{corBW-normality}
With the same assumptions as  in Lemma
\ref{lemma:BW-cont-pair-str-normality}, let us suppose that $\eta$
is invariant under $\psi$ and that $\psi$ is decomposable. Then
the $\bS^1$-invariant contact pair structure on the total space of the Boothby--Wang
fibration is normal if and only if the almost contact-symplectic
structure on the base is a normal almost contact structure.
\end{theorem}
\begin{proof}
Theorem \ref{th:BW-cont-pair-str} implies that $\psi$ is decomposable and $L_{Z_1} \phi =0$. By Theorem \ref{fifi} the
normality of the pair is equivalent to the integrability of $J$,
which follows from Theorem \ref{corBW}.
\end{proof}

We end this section with some examples:
\begin{example}
Taking for example a flat bundle where the base space is a closed
K\"ahler manifold with integral K\"ahler class (that is a
projective variety) and the fiber is a closed normal contact
manifold, yields a contact symplectic pair verifying the assumptions
of Theorem \ref{corBW-normality}.
\end{example}

\begin{example}
If the almost contact-symplectic structure $(\beta, \eta, \psi)$ has decomposable $\psi$ and is endowed with an associated metric
 as in Subsection \ref{subsec:alm-cont-sympl-str},
then the assumptions of Theorem \ref{corBW-normality} are satisfied.
\end{example}

\begin{example}
Using the double Boothby-Wang fibration over a closed manifold $B$ endowed with a
symplectic pair $(\omega_1 , \omega_2)$ such that $[\omega_i] \in
H^2 (B , \bZ)$ and a complex structure $J$ preserving the tangent spaces of the
foliations and compatible, on each leaf, with the symplectic form
induced by the pair, also gives an example for the Theorem \ref{corBW-normality}.

An interesting example of the former situation, already used in
\cite{BK}, is given by the quotient of a polydisc $\bH^2 \times
\bH^2$ by an irreducible lattice of the identity component of its
isometry group, where $\bH^2$ is the hyperbolic plane. In this case the pair is given by the K\"ahler
forms on each factor and the corresponding cohomology classes
are integral. More generally one could consider a product of $n$
copies of $\bH^2$.
\end{example}

\begin{remark}
Again with the Boothby--Wang fibration we obtain new constructions
of closed complex manifolds.
\end{remark}
\section*{Acknowledgement}
The first author is grateful to Dieter Kotschick for his help (in particular for suggesting Examples \ref{exD1} and \ref{exD2}) during the stay at the Centro de Giorgi in Pisa, in March 2008. The authors wish to thank Stefano Montaldo for his useful comments and Michel Goze for his encouragement. The authors also wish to thank the referee for his remarks that have improved the paper and, in particular, for pointing out references \cite{Abe, Blair2}.

%%%%%%%%%%%%%%%%%%%%%%%%%%%%%%
%%%%%%%%%%%%%%%%%%%%%%%%%%%%%
%%%%%%%%%%%%%%%%%%%%%%%%%%%%
%%%%%%%%%%%%%%%%%%%%%%%%%%%

\bibliographystyle{amsplain}

\begin{thebibliography}{999}

\bibitem{Abe} K. Abe, \textit{On a class of Hermitian manifolds},  Invent. Math.  \textbf{51}  (1979), no. 2, 103--121.
\bibitem{Bande1}
G.~Bande, \textsl{Formes de contact g{\'e}n{\'e}ralis{\'e},
couples de contact et couples contacto-symplectiques}, Th{\`e}se
de Doctorat, Universit{\'e} de Haute Alsace, Mulhouse 2000.

\bibitem{Bande2}
G.~Bande,
{\it Couples contacto-symplectiques},
Trans.~Amer.~Math.~Soc.~{\bf 355} (2003), 1699--1711.


\bibitem{BH}
G. Bande, A. Hadjar, \textit{Contact Pairs}, Tohoku Math. Journal
\textbf{57} (2005), no. 2, 247--260.

\bibitem{BH2}
G. Bande, A. Hadjar, \textit{Contact pair structures and associated
metrics}, Differential Geometry - Proceedings of the VIII International
Colloquium, World Sci. Publ., 2009, 266--275.


\bibitem{BGK}
G. Bande, P. Ghiggini, D. Kotschick, \textit{A stability theorem
for contact and symplectic pairs}, Int. Math. Res. Not., 2004, no.
68, 3673--3688.

\bibitem{BK}
G. Bande, D. Kotschick, \textit{The Geometry of Symplectic pairs},
Trans.~Amer.~Math.~Soc.~\textbf{358} (2006), no. 4, 1643--1655.

\bibitem{BG}
C.~Benson, C.~Gordon, {\it K\"ahler and symplectic structures on
nilmanifolds}, Topology {\bf 27} (1988), 513--518.

\bibitem{Blair}
D. E. Blair, \textit{The theory of quasi-Sasakian structures}, J.
Differential Geometry \textbf{1} (1967) 331--345.

\bibitem{blair3}
D. E. Blair, \textit{Geometry of manifolds with structural group $\mathcal{U}(n)\times \mathcal{O}(s)$}, J. Differential Geometry \textbf{4} (1970) 155--167.

\bibitem{Blairbook}
D. E. Blair, \textsl{Riemannian geometry of contact and symplectic
manifolds}, Progress in Mathematics, vol. 203, Birkh\"auser, 2002.

\bibitem{Blair2}
D. E. Blair, G. D. Ludden, K. Yano,  \textit{Geometry of complex manifolds similar to the Calabi-Eckmann manifolds},  J. Differential Geometry  \textbf{9}  (1974), 263--274.

\bibitem{Bor}
A. Borel, {\it Compact Clifford--Klein forms of symmetric spaces},
Topology {\bf 2} (1963), 111--122.

\bibitem{buch}
N. Buchdahl, {\it On compact K\"ahler surfaces}, Ann. Inst.
Fourier (Grenoble) \textbf{49} (1999), no. 1, 287--302.

\bibitem{BW}
W.~M.~Boothby, H.~C.~Wang, {\it On contact manifolds}, Ann.~of
Math.~{\bf 68} (1958), 721--734.

\bibitem{Geiges2}
H. Geiges, \textit{Normal contact structures on $3$-manifolds},
Tohoku Math. J. {\bf49} (1997), no. 3, 415--422.

\bibitem{Hata}
Y. Hatakeyama, \textit{Some notes on differentiable manifolds with
almost contact structures}, Tohoku Math. J. {\bf 15} (1963)
176--181.

\bibitem{KM}
D.~Kotschick, S.~Morita, {\it Signatures of foliated surface
bundles and the symplectomorphism groups of surfaces}, Topology
{\bf 44} (2005), no. 1, 131--149.

\bibitem{Ko}
D.~Kotschick, {\it Orientations and geometrisations of compact
complex surfaces}, Bull.~London Math.~Soc.~{\bf 29} (1997),
145--149.

\bibitem{Mori}
A. Morimoto, \textit{On normal almost contact structures}. J.
Math. Soc. Japan {\bf 15} (1963) 420--436.

\bibitem{nakagawa}
H. Nakagawa, \textit{$f$-structures induced on submanifolds in spaces, almost Hermitian or Kaehlerian},  K\=odai Math. Sem. Rep.  \textbf{18} (1966) 161--183.

\bibitem{N}
K. Nomizu, \textit{On the cohomology of compact homogeneous spaces
of nilpotent Lie groups}, Ann. of Math. \textbf{59} (1954),
531-538.

\bibitem{scott}
P. Scott, \textit{The geometries of 3-manifolds}, Bull. London Math. Soc. \textbf{15} (1983), 401--487.

\bibitem{Simon}
S. M. Salamon, \textit{Complex structures on nilpotent Lie
algebras}, J. Pure Appl. Algebra  \textbf{157}  (2001),  no. 2-3,
311--333.

\bibitem{yano}
K. Yano, \textit{On a structure defined by a tensor field $f$ of type $(1,\,1)$ satisfying $f\sp{3}+f=0$}, Tensor \textbf{14} (1963) 99--109.

\end{thebibliography}

\end{document}